\documentclass[a4paper, 12pt]{amsart}

\usepackage{caption}
\usepackage{graphicx}
\usepackage[T1]{fontenc}
\usepackage{verbatim}
\usepackage[colorlinks,linkcolor=blue,citecolor=blue]{hyperref}
\usepackage{graphics}
\usepackage{tikz}
\usetikzlibrary{cd}
\usepackage{color}
\usepackage{subcaption}
\usepackage[percent]{overpic}

\usetikzlibrary{calc}

\usepackage[top=2.5cm, bottom=2.5cm, left=2.5cm, right=2.5cm]{geometry}

\newtheorem{theorem}{Theorem}[section]
\newtheorem{lemma}[theorem]{Lemma}

\theoremstyle{definition}
\newtheorem{definition}[theorem]{Definition}

\newtheorem{question}[theorem]{Question}
\newtheorem{proposition}[theorem]{Proposition}

\newtheorem{remark}[theorem]{Remark}

\numberwithin{equation}{section}

\newcommand{\pretzvar}{k}
\newcommand{\cyclicgroup}[1]{\mathbb{Z}/{#1}}

\newcommand{\ex}{\begin{tikzpicture}[scale=0.13,baseline=-0.11cm,thick]\draw (-1,-1) -- (1,1);\draw (1,-1) -- (-1,1);\end{tikzpicture}}

\begin{document}

\title{Links all of whose cyclic branched covers are L-spaces}

\bibliographystyle{alpha}

\author{Ahmad Issa}
\address{Department of Mathematics, University of British Columbia, Canada}
\email{aissa@math.ubc.ca}

\author{Hannah Turner}

\address{Department of Mathematics, University of Texas at Austin}
\email{hannahturner@math.utexas.edu}

\begin{abstract} We show that for the pretzel knots $K_\pretzvar=P(3,-3,-2\pretzvar-1)$, the $n$-fold cyclic branched covers are L-spaces for all $n\geq 1$. In addition, we show that the knots $K_\pretzvar$ with $k\geq 1$ are quasipositive and slice, answering a question of Boileau-Boyer-Gordon. We also extend results of Teragaito giving examples of two-bridge knots with all L-space cyclic branched covers to a family of two-bridge links.

\end{abstract}

\maketitle

\section{Introduction}

Given a knot $K$ in $S^3$, it is an interesting problem to determine the set $$\mathcal{L}_{br}(K) = \{n \ge 1 : \mbox{the }n\mbox{-fold cyclic branched cover }\Sigma_n(K)\mbox{ is an L-space}\},$$ where $\Sigma_n(K)$ denotes the $n$-fold cyclic branched cover of $S^3$ over $K$, and an L-space is a rational homology $3$-sphere whose Heegaard Floer homology is as simple as possible. For example, when $K$ is the trefoil, $\mathcal{L}_{br}(K)=\{1,2,\ldots, 5\}$ (see \cite{GL14}) and when $K$ is the figure-eight knot, $\mathcal{L}_{br}(K)=\{1,2,\ldots\};$ see e.g.\ \cite{Pe09}. Evidence suggests the set  $ \mathcal{L}_{br}(K)$ always takes one of the two forms $\mathcal{L}_{br}(K) = \{1,2,\ldots\}$ or $\mathcal{L}_{br}(K) = \{1,\ldots,N\}$ for some $N$, and Boileau-Boyer-Gordon asked whether this holds in general. Some results are known, for example non-split alternating links have L-space double branched covers, and this is true more generally for Khovanov homology thin links \cite{OS-DBcover}.  It is known that the $n$-fold cyclic branched cover of $K$ is not an L-space for $n$ sufficiently large, provided that $K$ is a non-slice quasipositive knot \cite{BBG-QPL-space}, or $K$ is fibered with non-zero fractional Dehn twist coefficient \cite{Rob-Taut, KR-Approx} see also \cite{HM-Frac}. This class includes, for example, all L-space knots. 

One motivation for the study of $\mathcal{L}_{br}(K)$ comes from the L-space conjecture 
which, for an irreducible 3-manifold $M$, asserts that $M$ is an L-space if and only if $M$ has a non-left orderable fundamental group \cite{BGW13}.  Cyclic branched covers provide a natural class for which to study the L-space conjecture, and various results are known for left orderability of fundamental groups of branched covers; see for example \cite{BGW13, GL14, Hu15, Go17}.

In contrast to the case $\mathcal{L}_{br}(K) = \{1,\ldots,N\}$, knots for which $\mathcal{L}_{br}(K) = \{1,2,\ldots\}$, that is, $\Sigma_n(K)$ is an L-space for all $n \ge 1$, are less well understood. In fact, prior to our work, the only known examples was a family of $2$-bridge knots \cite{Pe09, Te14}. We prove the following theorem, giving a family of pretzel knot examples; see Figure \ref{pretzelandsymmetry} for our notation of pretzel knots.

\begin{theorem}
\label{pretzelsL-spaces}
The $n$-fold cyclic branched cover $\Sigma_n(P(3,-3,-2k-1))$ is an L-space for all integers $k$ and $n \ge 1$.
\end{theorem}

Boileau-Boyer-Gordon showed that if a quasipositive knot $K$ is not smoothly slice then $\Sigma_n(K)$ is not an L-space for $n$ sufficiently large, and asked the following question.

\begin{question}[\cite{BBG-QPL-space}] 
\label{BBGQ}
Does there exist a slice quasipositive knot all of whose branched cyclic covers are L-spaces?
\end{question}

Let $K_k$ denote the pretzel knot $P(3,-3,-2\pretzvar -1)$. These pretzel knots are well known to be slice; see Figure \ref{PretzelsQP} for a ribbon diagram. For $k\geq 1$, we show $K_k$ is a track knot, a notion introduced by Baader \cite{Ba-Track}.  Since track knots are quasipositive we obtain the following proposition, answering the above question in the affirmative.

\begin{proposition}\label{pretzelsqp}
The knots $K_k$ are slice and quasipositive for all $\pretzvar\geq 1$.
\end{proposition}

Before discussing these examples in more depth, we first look at a two bridge example. The figure-eight knot $K$ is the simplest knot for which $\Sigma_n(K)$ is an L-space for all $n \ge 1$. One way to see this is as follows. First observe that the figure-eight knot $K$ is a $2$-periodic knot, with quotient knot $Q$ and axis $A$ in the quotient as shown in Figure \ref{Figure8}. The $n$-fold cyclic branched cover $\Sigma_n(K)$ can be recovered by taking the $2$-fold branched cover over the axis $A$ and then taking the $n$-fold cyclic branched cover over the lift of $Q$. Reversing the order in which we take branched covers, $\Sigma_n(K)$ can also be obtained by taking the $n$-fold cyclic branched cover over $Q$, then the $2$-fold branched cover over the lift of $A$, which we denote $L_n$. Using the fact that there is an ambient isotopy interchanging the two components of $Q\cup A$, it is not difficult to obtain a diagram for $L_n$. The link $L_n$ is alternating and non-split, since it has a non-split alternating diagram \cite{Men-Alternating}. Thus, its double branched cover $\Sigma_2(L_n) \cong \Sigma_n(K)$ is an L-space for all $n \ge 1$ \cite{OS-DBcover}. 

\begin{figure}
\begin{tikzpicture}
	\node[anchor=south west,inner sep=0] (image) at (0,0) {\includegraphics[width=0.7\textwidth]{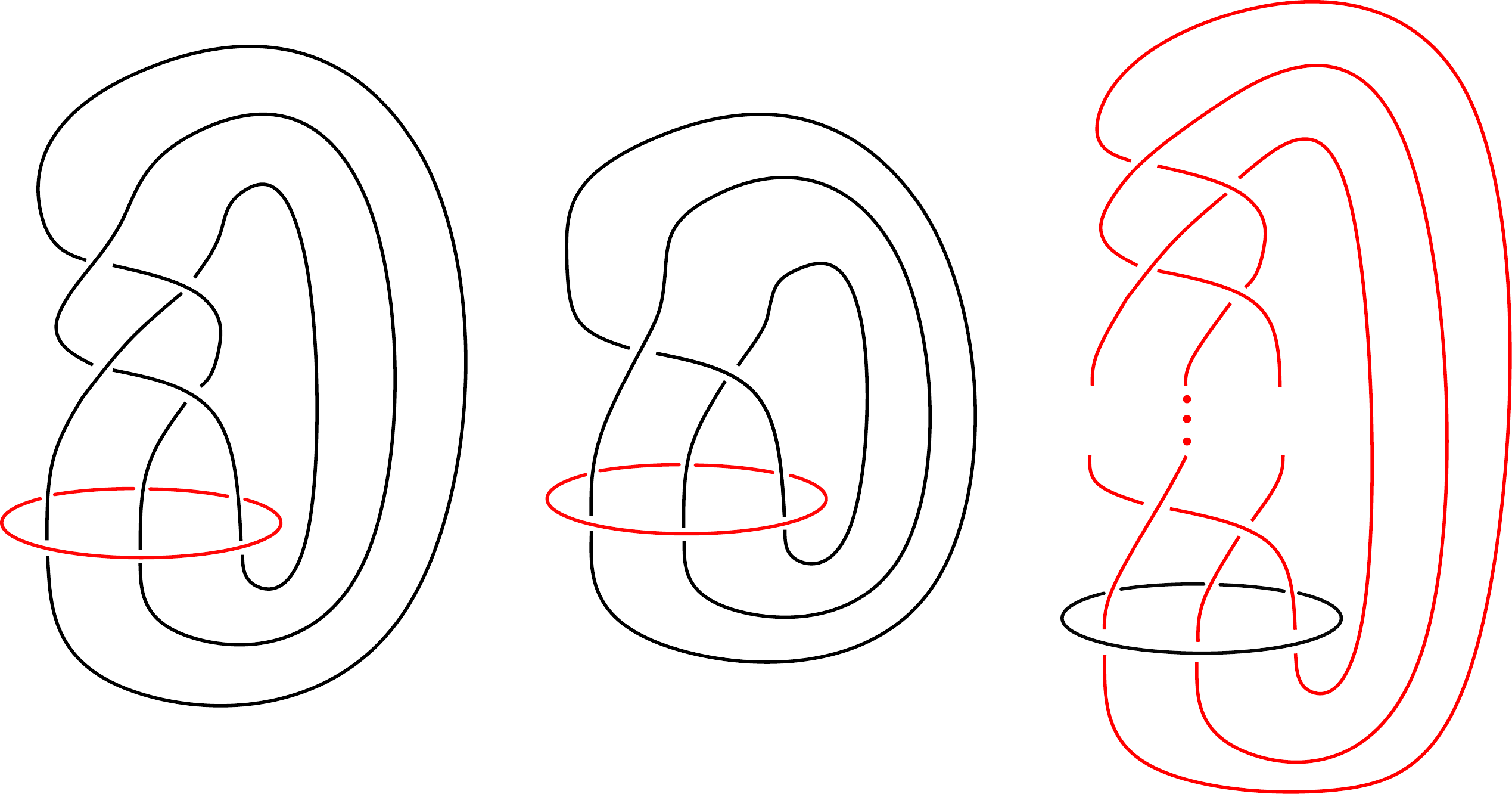}};
	\begin{scope}[x={(image.south east)},y={(image.north west)}]
        \node at (0,0.5) {$K$};
        \node at (0.35, 0.3){ $\textcolor{red}{A}$};
        \node at (0.35, 0.6){ $Q$};
        \node at (.7,.7){$\textcolor{red}{L_n}$};
    \end{scope}
\end{tikzpicture}
\caption{On the left we give a diagram for the figure-eight knot $K$ which after a planar isotopy has a rotational symmetry of order 2. Once we quotient by this symmetry we obtain the symmetric link $Q\cup A$. On the right, $L_n$ is an $n$-periodic link with quotient $A\cup Q$.}
\label{Figure8}
\end{figure}
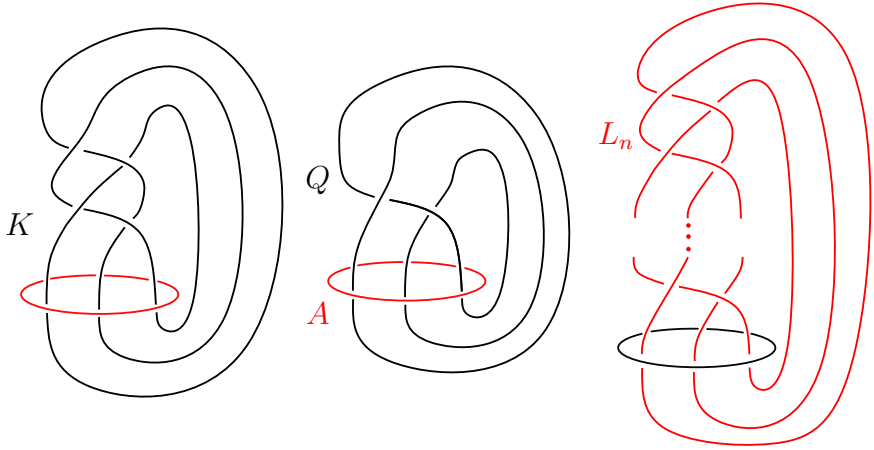

Peters proved that $\Sigma_n(K)$ is an L-space for two-bridge knots $K$ with fraction\footnote{A two-bridge knot with fraction $\frac{p}{q}$ has double branched cover the lens space $L(p,q)$. The figure-eight knot has fraction $\frac{5}{2} = [2, 2]^{+} = 2 + \frac{1}{2}$.} $[2a, 2b]^{+} := 2a + \frac{1}{2b}$ for all $a,b \ge 1$ and $n \ge 1$  \cite{Pe09}. Teragaito \cite{Te14} generalized these examples to two bridge knots with continued fractions of the form $[2a_1, \ldots, 2a_{2k}]^{+}$, where $a_1,\ldots,a_{2k} > 0$. In order to establish that $\Sigma_n(K)$ is the double branched cover of an alternating link in these cases, both Peters and Teragaito appeal to work of Mulazzani and Vesnin \cite{MV-Takahashi} which proceeds by means of certain surgery presentations of these manifolds. We remark that their results can be obtained directly using the $2$-periodic nature of two bridge links, generalizing the figure-eight knot case above. For any two-bridge link $L$ (with a particular orientation) we have $\Sigma_n(L)\cong \Sigma_2(L_n)$ for some link $L_n$; see e.g. \cite{MV-Cyclic2-bridge}. In certain cases, this link $L_n$ is non-split and alternating. This line of argument extends Teragaito's family to two bridge links \cite{Te14} as follows. We note that the special case of two-bridge torus links with fraction $[2a]^{+} = 2a$ and anti-parallel string orientations was first shown by Peters \cite{Pe09}.

\begin{theorem}
\label{twobridgebranched}

Let $L$ be a two-bridge link with fraction of the form $$[2a_1, 2a_2, \ldots, 2a_n]^{+} := 2a_1 + \cfrac{1}{2a_2 + \cfrac{1}{\ddots \,\,  + \cfrac{1}{2a_n}}},$$ where $a_1,\ldots,a_n$ are all positive integers. In the case that $L$ is a two-component link, we orient the link as in Figure \ref{twobridgestd}. Then $\Sigma_n(L)$ is an L-space for all $n\geq 1$.
\end{theorem}

More generally, whenever a link $L$ is $2$-periodic with quotient link the unknot, we can express $\Sigma_n(L)$ as the double branched cover over a link $L_n$. This is the case when $L$ is an odd pretzel knot with all tassel parameters odd integers. Thus, $\Sigma_n(P(3,-3,-2\pretzvar-1))$ is the double branched cover of a link $L_{n,k}$. However, unlike in the figure-eight knot case, the links $L_{n,k}$ are not, in general, alternating links. For example, $L_{2,1}$ is the pretzel knot $P(3,-3,-3)$ which is not even quasi-alternating \cite{MR2592726}; see also \cite{MR3825858}. We instead show that the links $L_{n,k}$ are two-fold quasi-alternating, a generalization of quasi-alternating links introduced by Scaduto and Stoffregen \cite{MR3760881}. Two-fold quasi-alternating links are $\cyclicgroup{2}$ reduced Khovanov thin \cite{MR3760881}. Thus $\Sigma_2(L_{n,k}) = \Sigma_n(P(3,-3,-2\pretzvar-1))$ is a Heegaard Floer L-space \cite{OS-DBcover} and a framed instanton L-space \cite{Sca-Instantons} with $\cyclicgroup{2}$ coefficients. In fact, we prove that the larger family of links $L(k_1,k_2,\ldots, k_n)$ in Figure \ref{LinkLnIntro} are all two-fold quasi-alternating. The link $L_{n,k}$ is given by $L(-k,-k,\ldots,-k)$, where $-k$ appears $n$ times.

\begin{theorem}
\label{LnTQA}
Let $L(k_1,k_2,\ldots, k_n)$ be the link in Figure \ref{LinkLnIntro}. Then $L$ is two-fold quasi-alternating for all integers $k_1, \ldots, k_n$.
\end{theorem}

\begin{figure}[!htbp]
\begin{tikzpicture}
	\node[anchor=south west,inner sep=0] (image) at (0,0) {\includegraphics[width=0.9\textwidth]{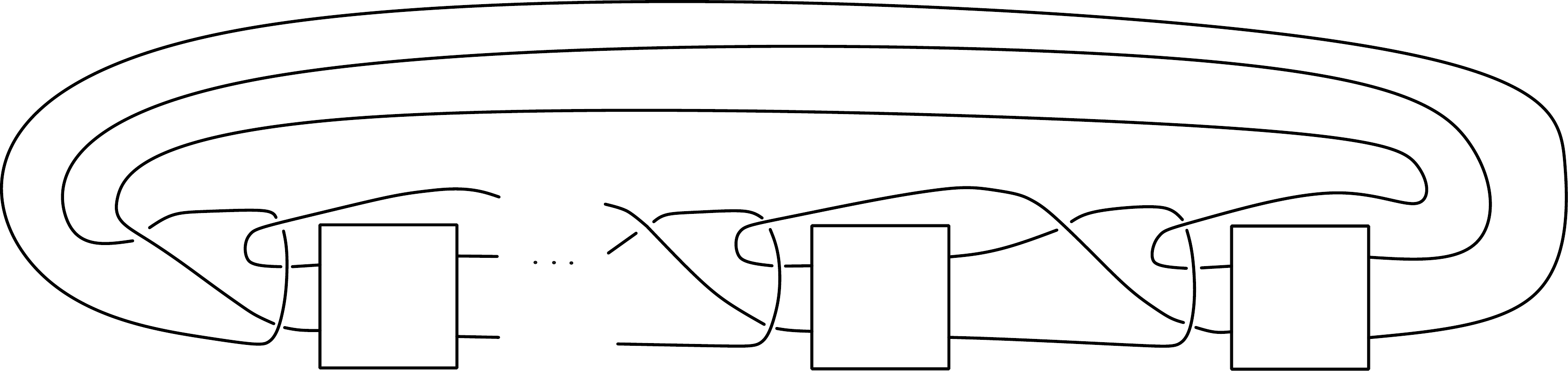}};
	\begin{scope}[x={(image.south east)},y={(image.north west)}]
        \node at (0.246,0.21) {\large $k_1$};
        \node at (0.565, 0.21){\large $k_{n-1}$};
        \node at (0.831, 0.21){\large $k_n$};
    \end{scope}
\end{tikzpicture}
\caption{The link $L(k_1,k_2,\ldots, k_n)$, where $k_1,\ldots,k_n$ are any integers. A twist box labelled $k_i$ denotes $k_i$ signed half-twists. See Figure \ref{pretzelandsymmetry} for an example of our signed twist conventions.}
\label{LinkLnIntro}
\end{figure}
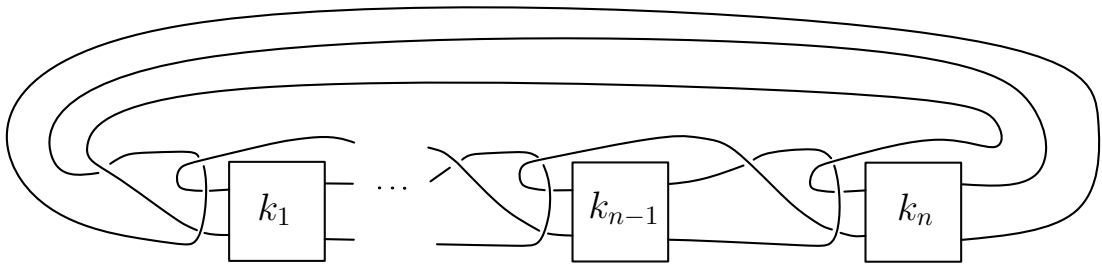

At the time of this writing, the links in Theorem \ref{pretzelsL-spaces} and Theorem \ref{twobridgebranched} include all currently known examples of knots and links all of whose branched cyclic covers are L-spaces. It is rather curious that in all of these examples, the knots satisfy the following properties. It is an interesting question whether any of these properties hold more generally.

\begin{enumerate}
\item\label{alexpoly} The roots of the Alexander polynomial $\Delta_K(t)$ are all real and positive \cite{LM-Zeros, Li-IntroKnots}.
\item\label{signature} The Tristram-Levine signature function $\sigma_\omega(K)$ is identically zero. This is implied by \eqref{alexpoly}.
\item\label{dinvariant} The $d$-invariant $d(\Sigma_n(K),\mathfrak{s}) = 0$ on the specific spin$^c$ structure $\mathfrak{s}$ described in \cite{LRS-Lefschetz}, for all $n\geq 2$.
\item The knot group $\pi_1(X_K)$ is biorderable \cite[Corollary 1.11]{JJ-BO}, \cite{JJ-BOP}.
\end{enumerate}

We remark that property \eqref{dinvariant} follows from property \eqref{signature} together with work of Lin-Ruberman-Saveliev establishing a relationship between the Tristram-Levine signatures and the Heegaard Floer $d$-invariant $d(\Sigma_n(K),\mathfrak{s})$ for a certain spin$^c$ structure when $K$ is a knot such that $\Sigma_n(K)$ is an L-space for all $n\geq 2$ \cite{LRS-Lefschetz}. We remark that the oriented two component links $L$ of Theorem \ref{twobridgebranched} all satisfy $|\sigma(L)|=1$ which is as small as possible among two component links whose double branched covers are rational homology spheres.

We end the introduction with the following question.
\begin{question} Does there exist a non $2$-periodic knot $K$ for which $\Sigma_n(K)$ is an L-space for all $n\geq 1$?
\end{question}

\subsection{Organization}

We exhibit $\Sigma_n(K_k)$ as the double-branched cover of a link $L_{n,k}$ in Section \ref{PretzelCovers}. In Section \ref{TQA} we give background on two-fold quasi-alternating links. We prove Theorems \ref{pretzelsL-spaces} and \ref{LnTQA} in Section \ref{OurLinksTQA}. We prove Proposition \ref{pretzelsqp} in Section \ref{Quasipositivity}. We conclude with a discussion of the case of two-bridge links and prove Theorem \ref{twobridgebranched}.

\subsection{Acknowledgements}
Conversations with Jonathan Johnson compelled us to study this family of pretzel knots. We would like to thank Cameron Gordon and Liam Watson for helpful feedback on an earlier draft. In addition, the second author thanks Cameron Gordon for many helpful conversations as well as his advice and support. The second author is supported by an NSF graduate research fellowship under grant no. DGE-1610403.

\section{Branched Cyclic Covers of $P(3,-3,-2\pretzvar-1)$}\label{PretzelCovers}

In this section we discuss $2$-periodic symmetries for pretzel links and apply these symmetries to express $\Sigma_n(P(3,-3,-2\pretzvar -1))$ as the double-branched cover of a link $L_{n,k}$. Let the pretzel link $P(p_1,p_2, \ldots, p_n)$ be defined by replacing the boxes in Figure \ref{pretzelandsymmetry} with $|p_i|$ half-twists with sign determined by the sign of $p_i$. We remark that cyclic permutations of the parameters do not change the link isotopy class. Then $K_\pretzvar=P(3,-3,-2\pretzvar -1)$, and in fact any pretzel link $P(p_1,p_2,\ldots, p_n)$ with each $p_i$ odd, admits an involution whose axis is disjoint from the link which we will now describe. Let $a$ be the horizontal line which passes through each twist box in the standard diagram for $K_k$ through the central crossing of each twist box. Then $K_k$ is rotationally symmetric about $a$; see Figure \ref{pretzelandsymmetry}. Rotating about this axis $a$ by an angle of $\pi$ defines an involution $\iota$ which preserves the knot setwise. The fixed set of this involution in $S^3$ is the axis $a$.

\begin{figure}[h!]
\begin{tikzpicture}[scale=1.0]
	\node[anchor=south west,inner sep=0] (image) at (0,0) {\includegraphics[width=.9\textwidth]{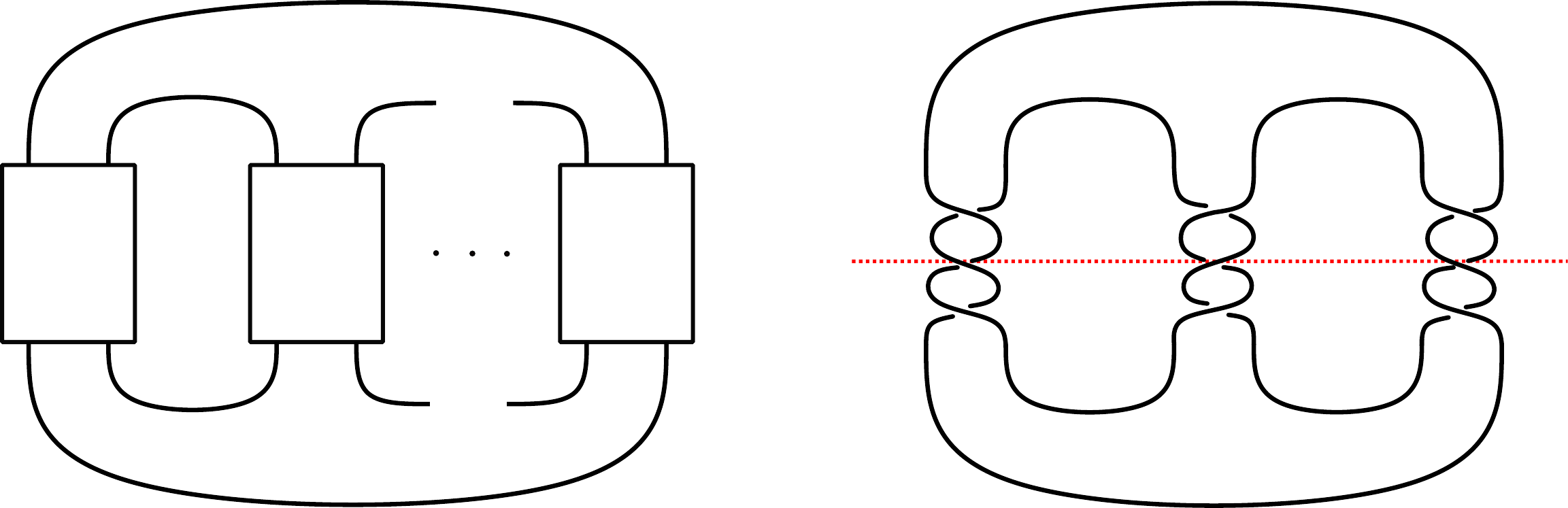} };
	\begin{scope}[x={(image.south east)},y={(image.north west)}]
	   \node at (0.045,.5) {\large $p_1$};
        \node at (.205, .5){\large $p_2$};
        \node at (.4, .5){\large $p_n$};
	\node[inner sep=0] (image) at (.98,.5) {\includegraphics[scale=.6]{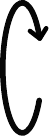} };
	\end{scope}
\end{tikzpicture}

\caption{The pretzel link $P(p_1,p_2,\ldots p_n)$ and a symmetry of $P(-3,3,-3)$.}
\label{pretzelandsymmetry}
\end{figure}
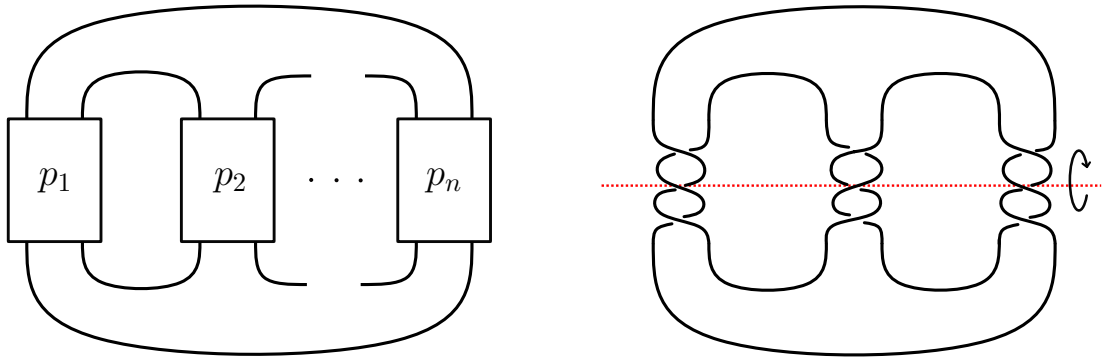

We are interested in the quotient by this involution. The quotient 3-manifold is $S^3$ and we denote the quotient knots by $\iota(S^3,K_\pretzvar,a)=(S^3,\overline{K_\pretzvar},\overline{a})$ pictured in Figure \ref{pretzelquotienttorus}. In fact, $\iota$ defines a branched covering map from $S^3$ to itself where the branching set is $\overline{a}$ and the branching is of index $2$.

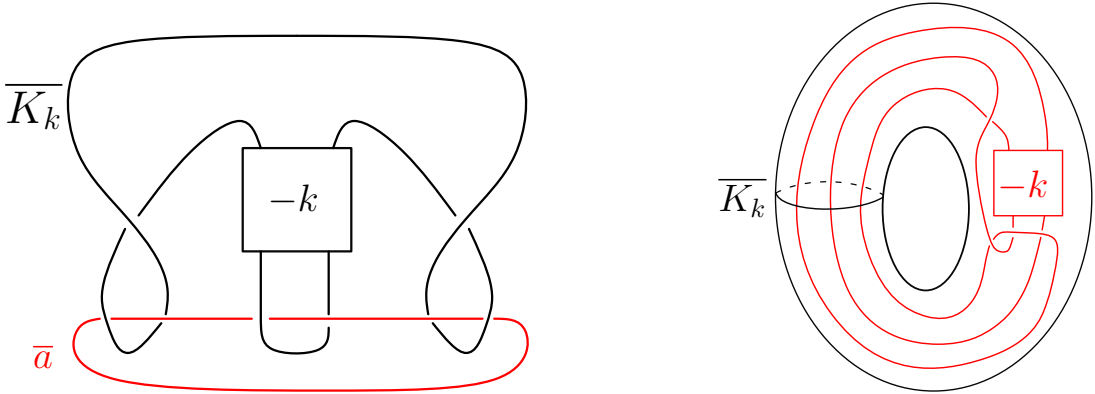
\begin{figure*}[hbt!]
        \centering
        \begin{subfigure}[b]{0.475\textwidth}
            \centering
          \begin{tikzpicture}
	\node[anchor=south west,inner sep=0] (image) at (0,0) {\includegraphics[width=0.8\textwidth]{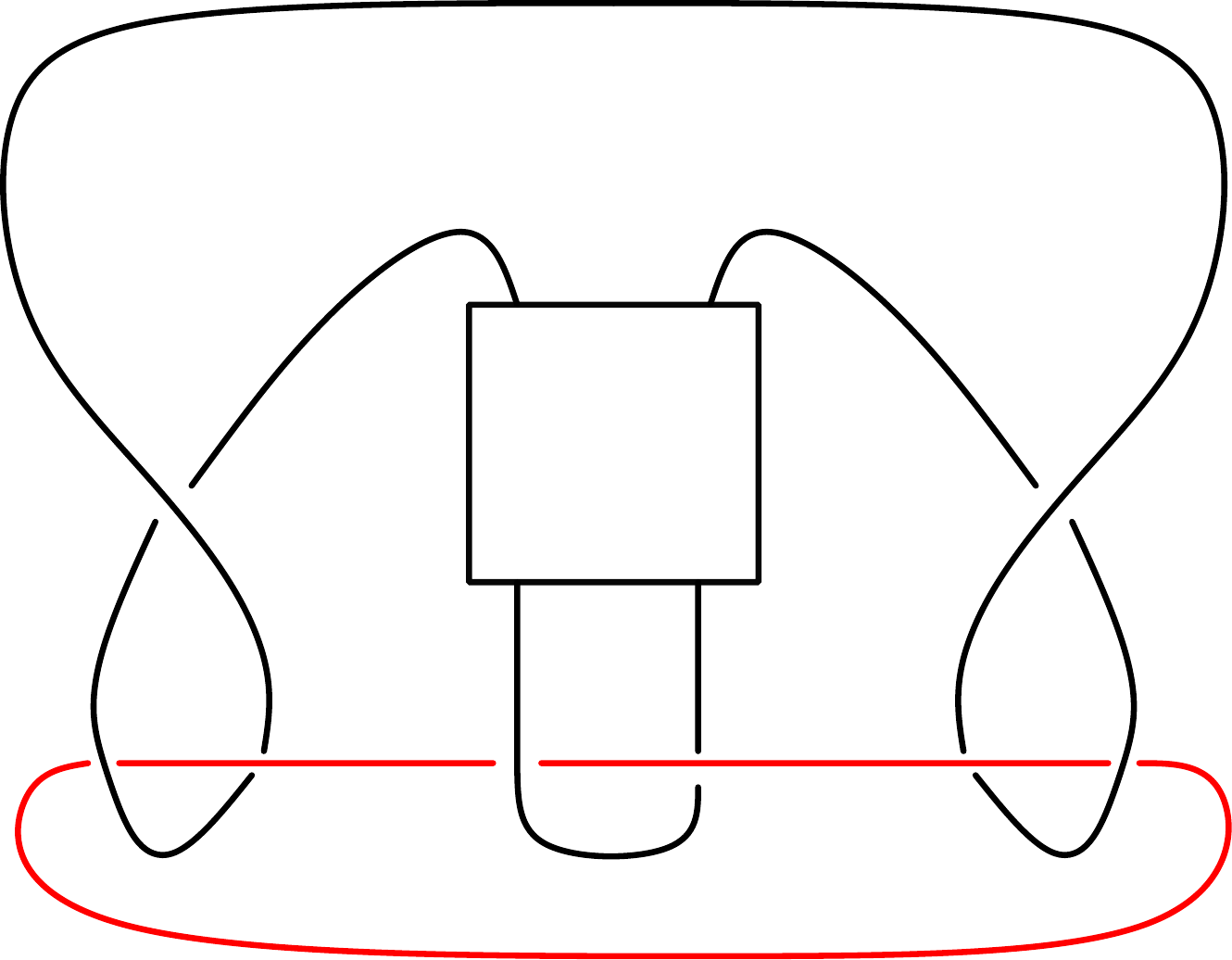}};
	\begin{scope}[x={(image.south east)},y={(image.north west)}]
	\node at (.49,.54){\large $-\pretzvar$};
	\node at (-.07,.8) {\Large $\overline{K_\pretzvar}$};
	\node at (-.05,.1){\large $\textcolor{red}{\overline{a}}$};
    \end{scope}
\end{tikzpicture}
               
            \label{Pretzelquotient}
        \end{subfigure}
        \hfill
        \begin{subfigure}[b]{0.475\textwidth}  
            \centering 
                      \begin{tikzpicture}
	\node[anchor=south west,inner sep=0] (image) at (0,0) {\includegraphics[width=0.55\textwidth]{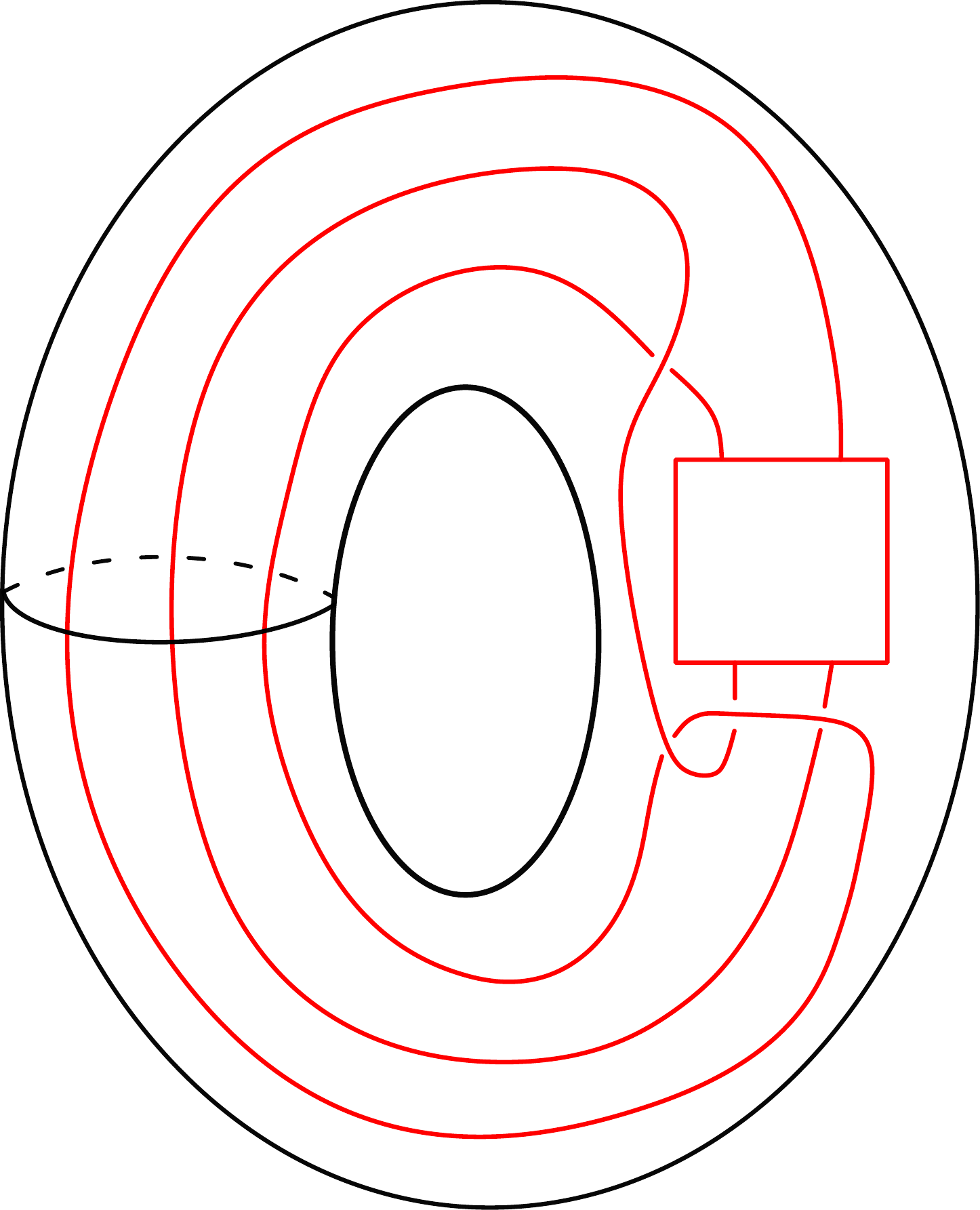}};
	\begin{scope}[x={(image.south east)},y={(image.north west)}]
	\node at (.78,.53){\large $\textcolor{red}{-\pretzvar}$};
	\node at (-.1,.5){\large $\overline{K_\pretzvar}$};
	\end{scope}
\end{tikzpicture}

        \end{subfigure}
       
        \caption{Let $K_\pretzvar=P(-3,-2\pretzvar-1,3)\cong P(3,-3, -2\pretzvar-1)$. There is an isotopy of $\overline{K_\pretzvar}\cup \overline{a}$ to a diagram where $\overline{a}$ sits inside a solid torus with $\overline{K_\pretzvar}$ its meridian} 
   \label{pretzelquotienttorus}
    \end{figure*}    

Taking any partition of the components of an oriented link $L$ into $l$ subsets so that $L=L_1\cup\cdots \cup L_l$, one can consider different branching indices $n_1,\ldots, n_l$ for each sublink. To denote the unique $\cyclicgroup{n_1}\times \cdots \times \cyclicgroup{n_l}$ branched cover obtained by this process, we will write $\Sigma_{n_1,n_2,\ldots, n_l}(L_1\cup L_2\cup\ldots \cup L_l)$. Thus, a feature of links $L$ which admit such a two-fold symmetry with axis $a$ avoiding the link, is that the $n$-fold branched cylic cover $\Sigma_n(L)$ can be expressed as the $\cyclicgroup{n}\times\cyclicgroup{2}$ branched cover of the quotient link $\overline{L}\cup \overline{a}$. The order of branching does not affect the homeomorphism type of the manifold $\Sigma_{n,2}(\overline{L}\cup \overline{a})$. Thus one can consider first branching (of index $n$) over $\overline{L}$, and then branching (of index 2) over the lift of $\overline{a}$ which we denote $L_n$; see Figure \ref{branchedtwoways}. This process and commutative diagram are well-known; see for instance \cite{MV-Cyclic2-bridge}. From this commuting diagram, we see that if $L_n=\phi_{n,1}^{-1}(\overline{a})$, then $ \Sigma_n(L)\cong \Sigma_{n,2}(\overline{L}\cup \overline{a})\cong \Sigma_2(L_n) $ can be expressed both as an n-fold cyclic branched cover over $L$ and a 2-fold branched cover over $L_n$.

\begin{figure}[h!]
\begin{tikzcd}[row sep=huge]
 &\ar[ld,"\rho_{n,1}"']  \Sigma_n(L)\cong\Sigma_{n,2}(\overline{L}\cup \overline{a}) \cong \Sigma_2(L_n)\ar[rd, "\phi_{1,2}"]&\\
\ar[rd,"\rho_{1,2}"'] \Sigma_{1,2}(\overline{L}\cup \overline{a})\cong (S^3,L,a) \hspace{0in} & & \hspace{0in}\Sigma_{n,1}(\overline{L}\cup \overline{a}) \ar[ld,"\phi_{n,1}"]\\
& (S^3, \overline{L},\overline{a})&
\end{tikzcd}

\caption{For 2-periodic link $L$, the $n$-fold branched cover can be expressed in two ways. The left-hand side of the diagram describes $\Sigma_{n,2}(\overline{L}\cup \overline{a})$ as first branching over $\overline{a}$, and then over $L$ (the lift of $\overline{L}$). On the right-hand side, we branch first over $\overline{L}$ and then over the lift of $\overline{a}$ which we denote $L_n$.}

\label{branchedtwoways}
\end{figure}

In the special case that $\overline{L}\cup\overline{a}$ is the union of two unknots, it is easy to understand the manifold $M=\Sigma_{n,1}(\overline{L}\cup \overline{a})$. In particular $M\cong S^3$ and it is possible to obtain a diagram for $(S^3, L_n)$ via the quotient diagram of $\overline{L}\cup \overline{a}$.

\begin{proposition}
\label{proplinkL_n}
$\Sigma_n(P(3,-3,-2\pretzvar-1))\cong \Sigma_2(L_{n,k})$ where $L_{n,k}$ is the link in Figure \ref{LinkLn}.
\end{proposition}

\begin{proof}
By appealing to the diagram in Figure \ref{branchedtwoways}, it suffices to describe the lift $\overline{a}$ of  $L_{n,\pretzvar}$ after branching of index $n$ over $\overline{K_\pretzvar}$ where $K_\pretzvar=P(3,-3,-2\pretzvar-1)$ and $\overline{K_\pretzvar}$ denotes its quotient under the involution described above and illustrated by Figure \ref{pretzelquotienttorus}.

The manifold $M$ we are interested in understanding is the $n$-fold cover of $\overline{K_\pretzvar}$ which is unknotted. The complement of $\overline{K_\pretzvar}$ is a solid torus $T$ whose lift is also a solid torus $\widetilde{T}$ in $M$. This solid torus $\widetilde{T}$ has meridian $\tilde{\mu}$ and longitude $\tilde{\lambda}$. If $\mu$ and $\lambda$ denote the meridian and longitude of $T$, then $\tilde{\lambda}$ is the lift of $\lambda$ while $\mu$ lifts to $n$ disjoint translates of $\tilde{\mu}$.

 Thus, it is easy to obtain a diagram for $L_{n,\pretzvar}$ the lift of $\overline{a}$. Figure \ref{LinkLn} is obtained by cutting open the torus of Figure \ref{pretzelquotienttorus} along $\overline{K_\pretzvar}$ and ``stacking'' $n$ copies of the corresponding cylinder and finally identifying the ends by the identity.
\end{proof}

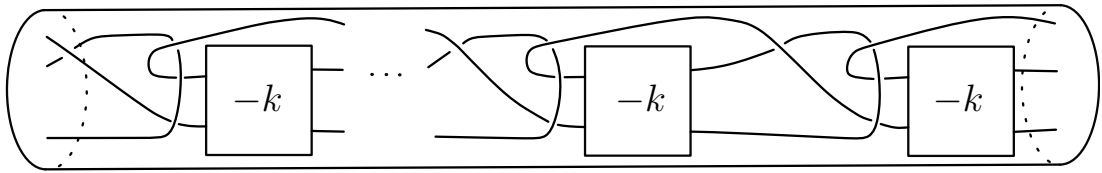
\begin{figure}[!htbp]
\begin{tikzpicture}
	\node[anchor=south west,inner sep=0] (image) at (0,0) {\includegraphics[width=0.9\textwidth]{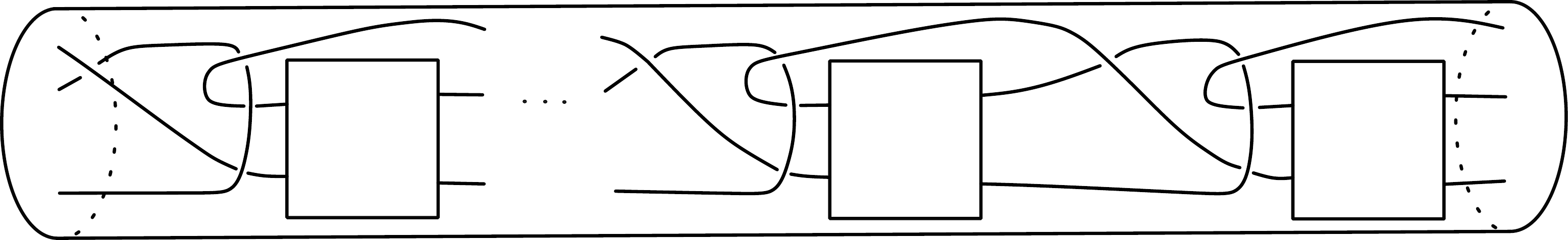}};
	\begin{scope}[x={(image.south east)},y={(image.north west)}]
        \node at (0.23,0.43) {\large $-\pretzvar$};
        \node at (0.58, 0.43){\large $-\pretzvar$};
        \node at (0.87, 0.43){\large $-\pretzvar$};
    \end{scope}
\end{tikzpicture}
\caption{The link $L_{n,\pretzvar}$ inside a torus (obtained by identifying the end disks of the cylinder by the identity map) sitting in $S^3$.}
\label{LinkLn}
\end{figure}

\section{Two-fold quasi-alternating links}\label{TQA}
In this section we recall the definition of two-fold quasi-alternating (TQA) links, as well as a few key facts which we will use. All of this material can be found in \cite{MR3760881}.

A \emph{marked link} is a pair $(L, \omega)$ where $L \subset S^3$ is a link and $\omega$ is an assignment of an element of $\cyclicgroup{2}$ to each component of $L$. If the number of components marked with $1 \in \cyclicgroup{2}$ is even (resp. odd) we say that $(L, \omega)$ is a \emph{two-fold marked link} (resp. \emph{odd marked link}). The trivial two-fold marking assigns $0 \in \cyclicgroup{2}$ to every component of $L$.

Let $D$ be a planar diagram representing a link $L$. An \emph{arc} of $D$ is a strand of $D$ that descends to an edge of the $4$-valent graph formed from $D$ upon forgetting its crossings. A \emph{marking of $D$} is an assignment $\tilde{\omega} : \Gamma(D) \rightarrow \cyclicgroup{2}$, where $\Gamma(D)$ is the set of arcs of $D$. We say $(D, \tilde{\omega})$ is \emph{compatible with} or \emph{represents} $(L, \omega)$ if 
$$\sum_{\gamma \in \Gamma(K)} \tilde{\omega}(\gamma) = \omega(K) \pmod{2},$$
for each component $K$ of $L$. This data can be packaged diagrammatically by placing an odd number of dots on each arc $\gamma$ for which $\tilde{\omega}(\gamma) = 1$. Using this diagrammatic interpretation, when we smooth $D$ at a crossing (see Figure \ref{fig:crossing_resolution}), we naturally obtain two marked diagrams $(D_0, \tilde{\omega}_0)$ and $(D_1, \tilde{\omega}_1)$ by counting dots mod $2$.

  \begin{figure}[h]
  \begin{overpic}[width=230pt]{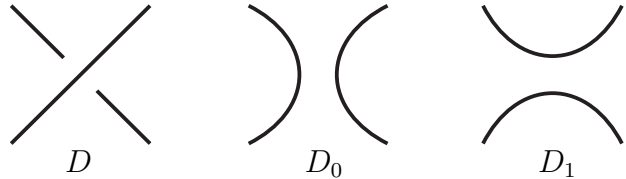}
    \put (9, 4) {$D$}
    \put (48, 4) {$D_0$}
    \put (86, 4) {$D_1$}
  \end{overpic}
\caption{$D$ and its two smoothings $D_0$ and $D_1$ in a neighbourhood of a given crossing.}
\label{fig:crossing_resolution}
\end{figure}

\begin{definition} The set of \emph{two-fold quasi-alternating (TQA) two-fold marked links}, denoted $\mathcal{Q}$, is the smallest set of two-fold marked links satisfying the following:
\begin{enumerate}
  \item The unknot with its unique trivial two-fold marking data is in $\mathcal{Q}$.
  \item Any two-fold marked link that splits into two odd-marked links is in $\mathcal{Q}$.
  \item Let $(D, \tilde{\omega})$ be a marked diagram representing $(L, \omega)$ such that the two smoothings $(D_0, \tilde{\omega}_0)$ and $(D_1, \tilde{\omega}_1)$ at a crossing represent marked links $(L_0, \omega_0)$ and $(L_1, \omega_1)$, respectively (see Figure \ref{fig:crossing_resolution}). If
    \begin{itemize}
      \item both smoothings $(L_0, \omega_0)$ and $(L_1, \omega_1)$ are in $\mathcal{Q}$, and
      \item $\det(L) = \det(L_0) + \det(L_1)$
      \end{itemize}
      then $(L, \omega)$ is in $\mathcal{Q}$.
    \end{enumerate}
    We say that a link $L$ is \emph{two-fold quasi-alternating} (TQA) if for the trivial marking $\omega$, we have $(L, \omega) \in \mathcal{Q}$.
  \end{definition}

All non-split alternating links are TQA. Finally, the key fact about two-fold quasi-alternating links which we will use is the following; see \cite[Corollary 1]{MR3760881}.

\begin{theorem} If a link $L$ is TQA, then $\Sigma_2(L)$ is a Heegaard Floer L-space with $\cyclicgroup{2}$ coefficients.
\end{theorem}

\section{$L_{n,\pretzvar}$ are two-fold quasi-alternating}\label{OurLinksTQA}
The goal of this section is to show that all the links $L_{n,\pretzvar}$ in Figure \ref{LinkLn} are two-fold quasi-alternating. This implies that their double branched covers are L-spaces. We will in fact show that a broader family of links is TQA. We first define the links of interests.

\begin{definition}
  Let $T(k)$ denote the $3$-tangle shown in Figure \ref{fig:tk_three_tangle2}, where $k \in \mathbb{Z} \cup \{\infty\}$. If $T_1$ and $T_2$ are $3$-tangles, denote by $T_1 \oplus T_2$ the $3$-tangle obtained by stacking the two tangles, with $T_1$ to the left of $T_2$; see for example Figure \ref{fig:tk_plus_tinfty}. Let $L(k_1, \ldots, k_n) \subset S^3$ denote the link given by the closure of the tangle $T(k_1) \oplus T(k_2) \oplus \cdots \oplus T(k_n)$, where $k_i \in \mathbb{Z} \cup \{\infty\}$ for all $i$.
\end{definition}

The link $L_{n,\pretzvar}$ in Figure \ref{LinkLn} is the link $L(-\pretzvar, -\pretzvar, \ldots, -\pretzvar)$. We will show that these links are all TQA.

\begin{figure}[h]
  \begin{overpic}[height=80pt, grid=false]{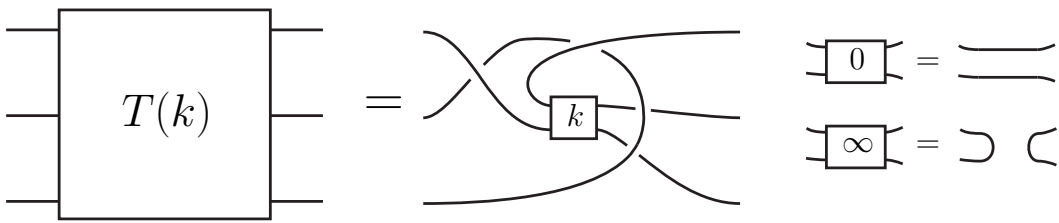}
    \put (11, 9) {\fontsize{18}{18}$T(k)$}
    \put (53.3, 9) {$k$}
    \put (34, 10) {\fontsize{20}{20}$=$}
    \put (80.2, 14.5) {$0$}
    \put (79.5, 6.5) {$\infty$}
    \put (86.5, 6.5) {$=$}
    \put (86.5, 14.5) {$=$}
  \end{overpic}
  \caption{The tangle $T(k)$ where $k$ is the number of signed half-twists. See Figure \ref{pretzelandsymmetry} for an example of our signed twist conventions. We also allow the case $k=\infty$ which by convention is as shown.}
  \label{fig:tk_three_tangle2}
\end{figure}

\begin{lemma}\label{lemma:lk_is_a_knot}The link $L(k_1, k_2, ..., k_n)$ is a knot provided no $k_i=\infty$.
\end{lemma}
\begin{proof} The link $L = L(k_1, k_2, ..., k_n)$ is the closure of $T(k_1) \oplus \cdots \oplus T(k_n)$. To count the number of components of $L$, we only need to keep track of the way each tangle $T(k_i)$ connects the $6$ endpoints of the tangle, which is determined by the parity of $k_i$. If $T$ and $T'$ are $3$-tangles, write $T \sim T'$ if $T$ and $T'$ are homotopic, in other words, they have the same connectivity of endpoints and number of connected components. One can check the four cases: $T(0) \oplus T(1) \sim T(1)$, $T(0) \oplus T(0) \sim T(0)$, $T(1) \oplus T(0) \sim T(0)$ and $T(1) \oplus T(1) \sim T(1)$. Thus, $T(k_1) \oplus \cdots \oplus T(k_n) \sim T(k)$, where $k \in \{0,1\}$. Finally, the closure of both $T(0)$ and $T(1)$ are knots, hence $L$ is a knot.
\end{proof}

\begin{lemma}\label{lemma:split_base_case} Let $L = L(k_1,\ldots,k_n)$ with $k_i \in \mathbb{Z} \cup \{\infty\}$ for all $i$. If $k_i = \infty$ for exactly one value of $i \in \{1,\ldots,n\}$, then $L$ is the two component unlink.

\end{lemma}

\begin{proof} Notice that the link $L = L(k_1,\ldots,k_n)$ is unchanged under any cyclic permutation of the parameters $(k_1,\ldots,k_n)$. Hence, we may assume that $k_n = \infty$. Figure \ref{fig:tk_plus_tinfty} shows that $T(k) \oplus T(\infty)$ is isotopic to $T(\infty)$ as tangles, provided $k \neq \infty$. Applying this $n-1$ times, we see that $T(k_1) \oplus \cdots \oplus T(k_n)$ is the same as $T(\infty)$. Hence, $L$ is the closure of $T(\infty)$ which is the two component unlink.
\end{proof}

\begin{figure}[h!]
  \begin{overpic}[height=90pt, grid=false]{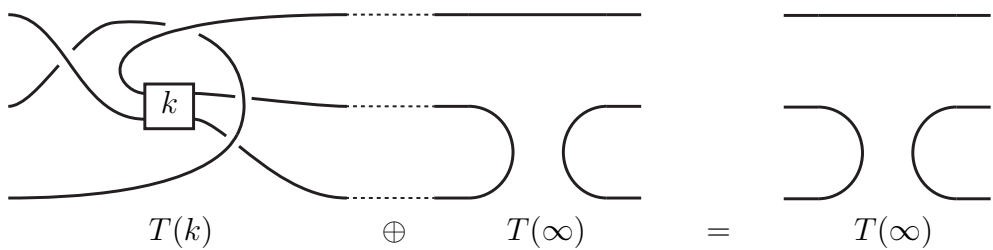}
    \put (14.3, 1) {$T(k)$}
    \put (50.7, 1) {$T(\infty)$}
    \put (38, 1) {$\oplus$}
    \put (71, 1) {$=$}
    \put (86, 1) {$T(\infty)$}
    \put (15.5, 14) {$k$}
  \end{overpic}
  \caption{The identity $T(k) \oplus T(\infty) = T(\infty)$.}
  \label{fig:tk_plus_tinfty}
\end{figure}

\begin{lemma}\label{lemma:alt_base_case} The links of the form $L(0,0,\ldots,0)$ are all non-split alternating knots.
\end{lemma}

\begin{proof} The link $L = L(0,0,\ldots,0)$ is the closure of $T(0) \oplus \cdots \oplus T(0)$. It is a knot by Lemma \ref{lemma:lk_is_a_knot}. If $T$ is a tangle conjugate to $T(0)$ then $L$ is also given by the closure of $T \oplus \cdots \oplus T$. Figure \ref{fig:alternating_tangle} shows that we can conjugate $T(0)$ to an alternating tangle $T$, from which we see that the closure of $T \oplus \cdots \oplus T$ has a non-split alternating diagram.
\end{proof}

  \begin{figure}[h]
  \begin{overpic}[width=\textwidth, grid=false]{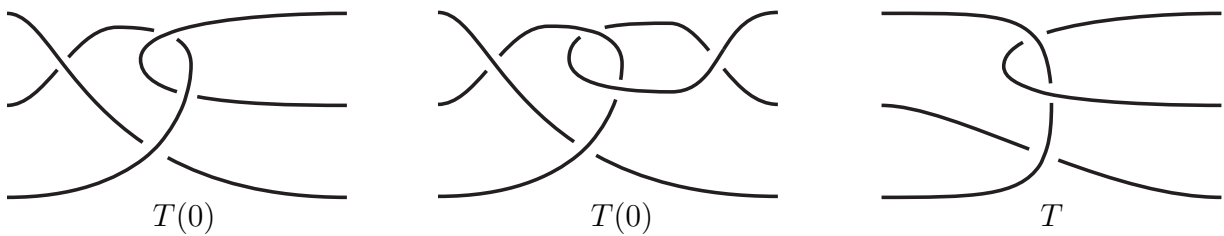}
    \put (12, 1) {$T(0)$}
    \put (48, 1) {$T(0)$}
    \put (85, 1) {$T$}
  \end{overpic}
  \caption{The alternating tangle $T$ is obtained from $T(0)$ by conjugating by a tangle interchanging the top two strands. The middle diagram of $T(0)$ is obtained from the one on the left by an isotopy. }
  \label{fig:alternating_tangle}
\end{figure}

\begingroup
\def\thetheorem{\ref{LnTQA}}
\begin{theorem} The knot $L(k_1, \ldots, k_n)$, where $k_1, \ldots, k_n$ are integers is two-fold quasi-alternating.
\end{theorem}
\addtocounter{theorem}{-1}
\endgroup

\begin{proof} We prove this by induction on $N = |k_1| + \cdots + |k_n|$. Let $L = L(k_1, \ldots, k_n)$. If $N = 0$ then $k_i = 0$ for all $i$ and $L$ is a non-split alternating knot by Lemma \ref{lemma:alt_base_case}, and hence is TQA. Assume that $N > 0$ and let $i \in \{1,\ldots,n\}$ be an index such that $|k_i| > 0$.
 
  For convenience, let $\varepsilon = \mbox{sign}(k_i) \in \{\pm 1\}$. Pick any crossing in the twist box labeled $k_i$. Smoothing $L$ at the crossing results in two links: $L_0 = L(k_1, \ldots, k_{i-1}, \varepsilon(|k_i| - 1), k_{i+1}, \ldots, k_n)$ and $L_1 = L(k_1, \ldots, k_{i-1}, \infty, k_{i+1}, \ldots, k_n)$. By Lemma \ref{lemma:split_base_case}, $L_1$ is the two component unlink. Let $D$ denote our diagram for $L$, and take any two arcs of $D$ which, after smoothing the crossing, belong to the two distinct components of $L_1$. Define a marking on $D$ which assigns $1 \in \cyclicgroup{2}$ (a single dot) to each of these two arcs, and $0$ (no dots) to all other arcs. Since $L$ is a knot by Lemma \ref{lemma:lk_is_a_knot}, this marking represents the trivial marking of $L$ (the two dots cancel $\mbox{mod }2$). Smoothing the crossing, with this choice of marking, we obtain marked diagrams representing:

\begin{itemize}
\item $L_0$ with its trivial marked diagram since $L_0$ is a knot by Lemma \ref{lemma:lk_is_a_knot}. This marked link is TQA by the induction hypothesis.
\item The two component unlink $L_1$, where each component is assigned $1$. This is TQA as it splits into two odd marked unknots.
\end{itemize}
We note that the determinant condition $\det(L) = \det(L_0) + \det(L_1)$ is automatically satisfied since $\det(L_1) = 0$. Hence, $L$ (with its trivial marking) is TQA.

\end{proof}

Since the links $L_{n,\pretzvar}$ in Figure \ref{LinkLn} are the links of the form $L(-\pretzvar, -\pretzvar, \ldots, -\pretzvar)$, the above theorem implies that they are all TQA. Combining this with Proposition \ref{proplinkL_n}, we obtain the following.

\begingroup
\def\thetheorem{\ref{pretzelsL-spaces}}
\begin{theorem}
The branched cover $\Sigma_n(P(3,-3,-2\pretzvar-1))$ is an L-space for all integers $k$ and $n \ge 1$.
\end{theorem}
\addtocounter{theorem}{-1}
\endgroup

\section{Quasipositivity and sliceness of the knots $P(3,-3,-2\pretzvar-1)$}\label{Quasipositivity}

In this section we show that the knots $P(3,-3,-2\pretzvar -1)$ are quasipositive. The argument used also shows the well-known result that these pretzels knots are slice. Let $B_n$ denote the Artin braid group on $n$ strands and let $\sigma_i$ be the standard Artin generators for $1\leq i\leq n$.

We now recall the definition of a quasipositive link. A braid $\beta\in B_n$ is called quasipositive if it can be written as a product of conjugates of the Artin generators: $\beta=\prod_k w_k\sigma_{i_k}w_k^{-1}$. A link in $L$ in $S^3$ is quasipositive if it is the closure of a  quasipositive braid.

Quasipositivity is useful in determining properties of the branched cyclic covers of the knots in $S^3$ as the following theorem shows.

\begin{theorem}[\cite{BBG-QPL-space}]
Suppose $K$ is a non-slice quasipositive knot. Then there is an $N=N(K)$ so that $\Sigma_n(K)$ is not an L-space for all $n\geq N$.
\end{theorem}

The following proposition shows that the pretzel knots which were shown in Section \ref{OurLinksTQA} to have all L-space branched covers are also quasipositive. Note the knots in the proposition are slice; see Figure \ref{PretzelsQP} for a ribbon diagram. This answers Question \ref{BBGQ}.

\begingroup
\def\thetheorem{\ref{pretzelsqp}}
\begin{proposition}
The knots $P(3,-3,-2\pretzvar-1)$ are slice and quasipositive for $\pretzvar\geq 1$.
\end{proposition}
\addtocounter{theorem}{-1}
\endgroup

In order to prove this, we use the notion of a track knot defined by Baader \cite{Ba-Track}.

\begin{figure}[h!]
\begin{tikzpicture}
	\node[anchor=south west,inner sep=0] (image) at (0,0) {\includegraphics[width=0.6\textwidth]{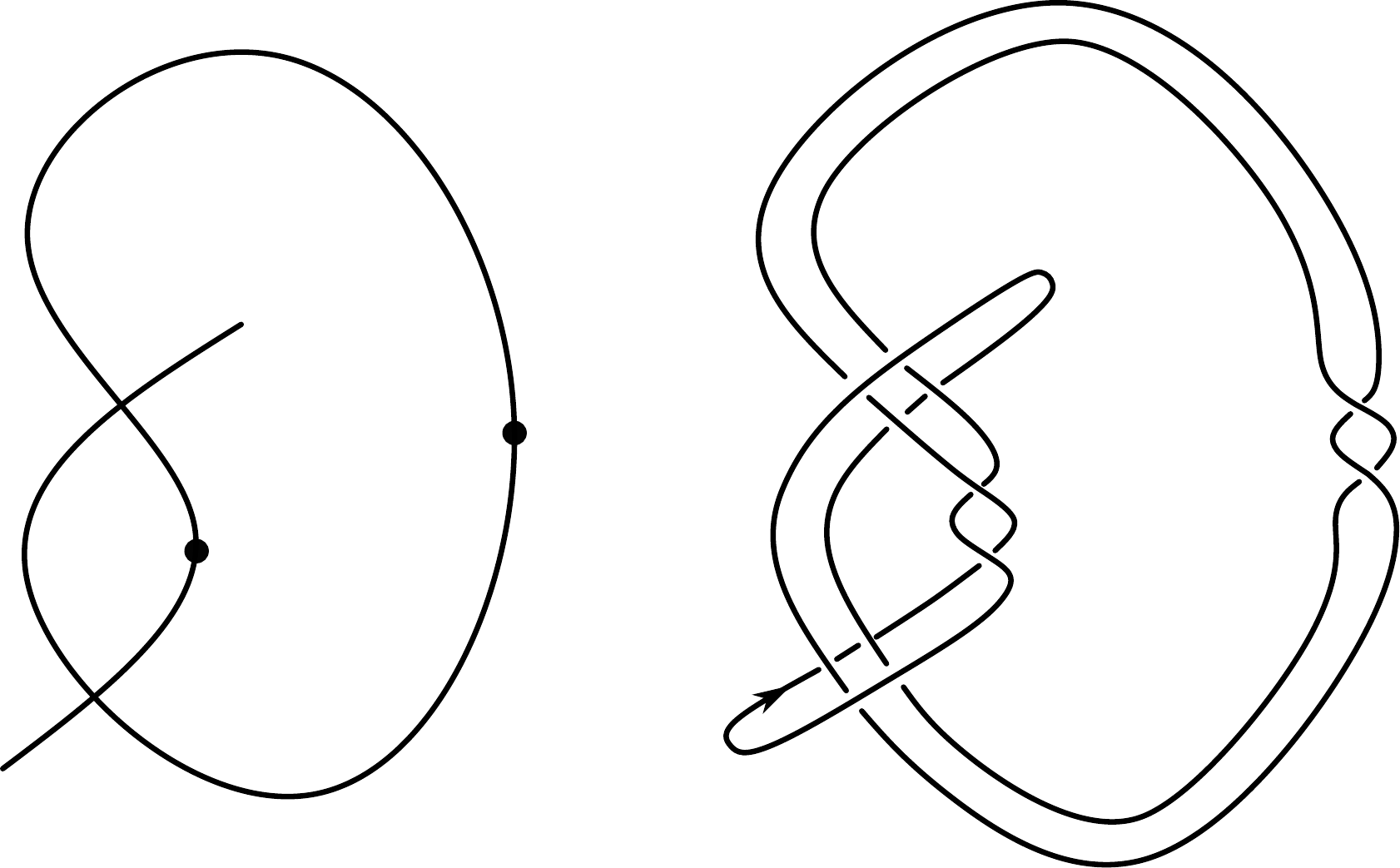} };
	\begin{scope}[x={(image.south east)},y={(image.north west)}]
        \node at (-.025,.54) { $(b,3\pi/2)$};
        \node at (-.01, 0.21){$(b,\pi)$};
    
    \end{scope}
\end{tikzpicture}
\caption{An immersed labeled interval and the corresponding track knot, which happens to be isotopic to $P(3,-3,-3)$.}
\label{trackexample}

\end{figure}

\begin{definition}
Let $C$ be the image of a generic immersion $i:[0,1]\to \mathbb{R}^2$ with a labeling at each double point by a letter in  $\{a,b,c,d\}$, and an angle in $\{0,\pi/2,\pi, 3\pi/2\}$; here we assume the diagram has been isotoped so that double points locally look like $\ex$ in the plane (not some arbitrary rotation of this picture). Finally, specify points $p_1,p_2,\ldots, p_r$ on the connected components of $C-\{\mbox{double points}\}$ such that each connected component of $C-\{p_1,p_2,\ldots,p_r\}$ is contractible. We will call $C$ a \textit{labeled immersed interval}.
\end{definition}

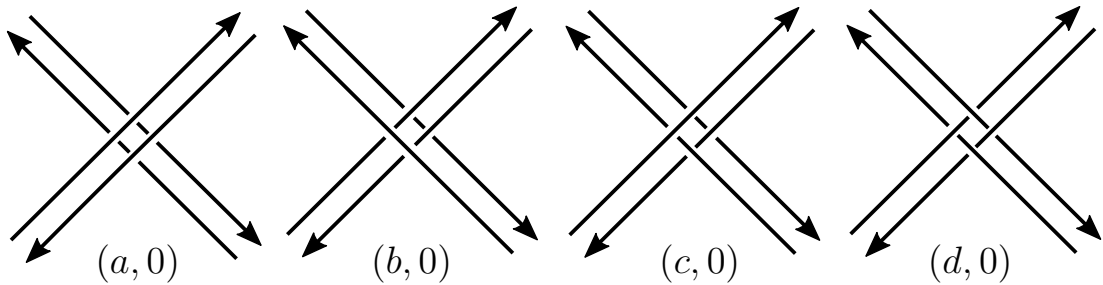
\begin{figure}[!htbp]
\begin{tikzpicture}
	\node[anchor=south west,inner sep=0] (image) at (0,0) {\includegraphics[width=0.9\textwidth]{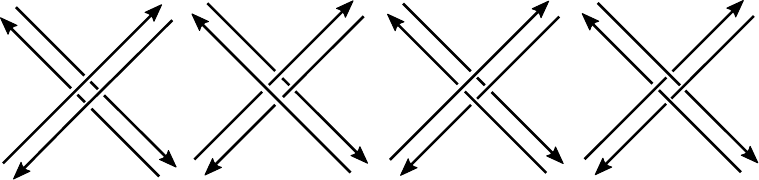}};
	\begin{scope}[x={(image.south east)},y={(image.north west)}]
        \node at (0.12,0) {\large $(a,0)$};
        \node at (0.633, 0){\large $(c,0)$};
        \node at (0.37, 0){\large $(b,0)$};
        \node at (0.88, 0){\large $(d,0)$};
    \end{scope}
\end{tikzpicture}
\caption{The process defining a knot from a labeled immersed curve replaces a double point labeled $(a,0), (b,0),(c,0)$ and $(d,0)$ with the corresponding crossing patterns above. If the angle at the crossing is not $0$, rotate the corresponding diagram above by the angle specified in the counter-clockwise direction.}
\label{trackcrossings}
\end{figure}

From such an interval one can associate a knot as follows (see Figure \ref{trackexample}  for an example). Draw an immersed interval parallel to $C$ and join the two intervals by an arc at each end of $C$, this forms an immersed band following $C$ which we will call $B_C$.  Orient $\partial B_C$ counter-clockwise. Each double point of $C$  corresponds to four self-intersections of $\partial B_C$, which we replace with over-and under-crossings according to the labeling and angle of $C$ at that double point as shown in Figure \ref{trackcrossings}. Replace each point $p_i$ with a full-twist oriented to introduce positive crossings.

\begin{definition}
A knot obtained from a labeled immersed interval by the above procedure is called a \textit{track knot}.
\end{definition}

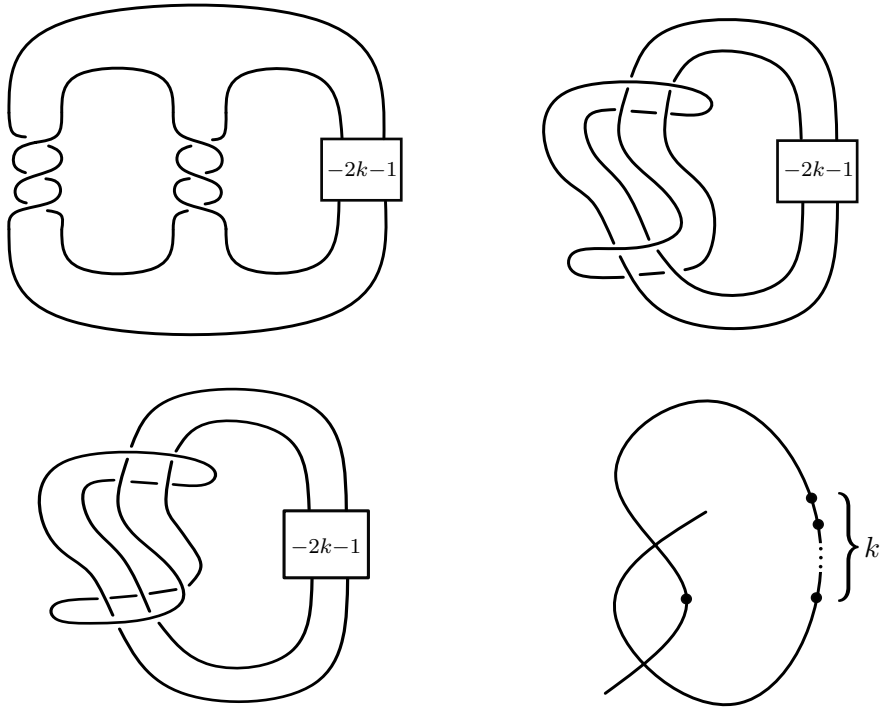
\begin{figure}[h!]
\begin{tikzpicture}
	\node[inner sep=0] (image1) at (0,0) {\includegraphics[width=0.7\textwidth]{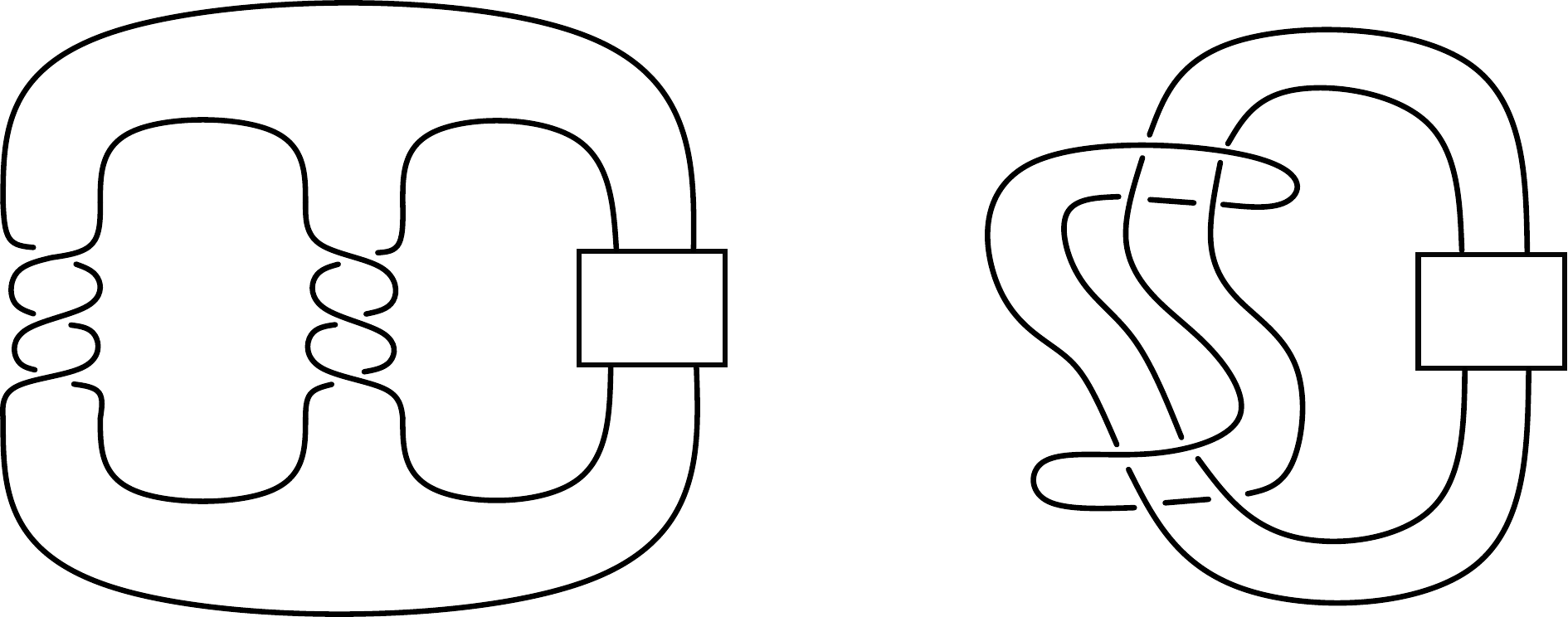}};
	\node[inner sep=0] (image2) at (0,-5) {\includegraphics[width=0.65\textwidth]{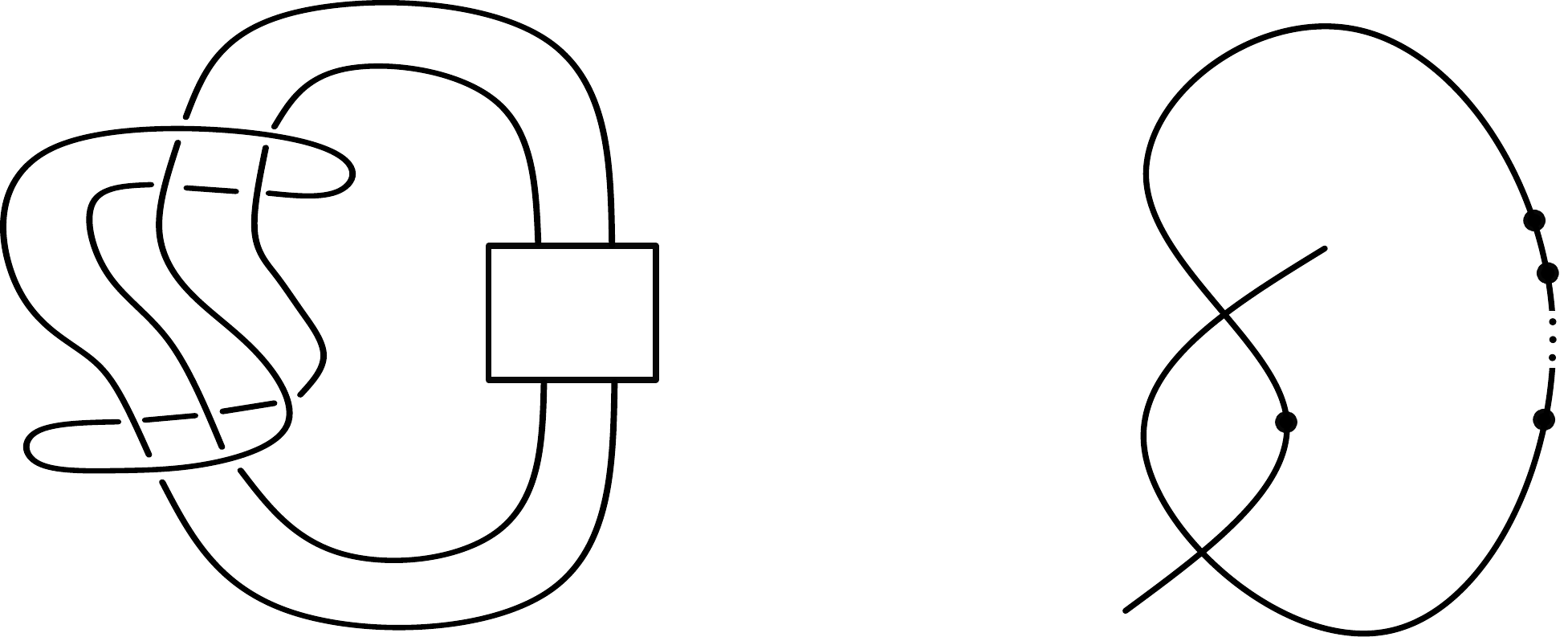}};
		
	\begin{scope}[x={($(image1.east)-(image1)$)},y={($(image1.south)-(image1)$)},shift=(image1)]
	\node at (-.166,0) [rotate=0]{\tiny $-2\pretzvar{-}1$};
	\node at (.905,0) [rotate=0]{\tiny $-2\pretzvar{-}1$};
	\end{scope}
	
	\begin{scope}[x={($(image2.east)-(image2)$)},y={($(image2.north)-(image2)$)},shift=(image2)]
	\node at (-.27,0) [rotate=0]{\tiny $-2\pretzvar{-}1$};
	\node at (1.05,0) {\resizebox{0.15in}{0.43in}{\large $\}$}};
	\node at (1.11,0) {\small $\pretzvar$};
	 \end{scope}
\end{tikzpicture}
\caption{An isotopy of the pretzel knot $P(3,-3,-2\pretzvar -1)$ which in particular exhibits ribbon disks for each knot. After redistributing a half twist from the twist box in the final pictured knot diagram by isotopy, we obtain the bottom right picture which gives a labeled immersed interval for $P(3,-3,-2\pretzvar -1)$ realizing these knots as track knots.}
\label{PretzelsQP}
\end{figure}

\begin{theorem}[\cite{Ba-Track}]
\label{trackqp}
Track knots are quasipositive.
\end{theorem}

\begin{remark} Baader's proof that track knots are quasipositive describes an algorithm to obtain a quasipositive braid word for a track knot, though in general it will not be of minimal braid index.
\end{remark}

 Figure \ref{PretzelsQP} shows that the knots $P(3,-3,-2\pretzvar -1)$ are track knots and exhibits a slice disk for them. We then obtain the following proposition.

\begin{proposition}
The pretzel knots $P(3,-3,-2\pretzvar-1)$ are slice track knots for $\pretzvar\geq 1$.
\end{proposition}

Proposition \ref{pretzelsqp} now follows from Theorem \ref{trackqp}.

\section{Branched cyclic covers of two-bridge links}\label{TwoBridge}
In this section we prove Theorem \ref{twobridgebranched} giving a family of two bridge links, all of whose cyclic branched covers are L-spaces.

We first briefly recall some background on two bridge links. The two bridge link associated with the fraction $\frac{p}{q} \in \mathbb{Q}$, where $p > q > 0$ are coprime is the unique link $L_{p/q}$ in $S^3$ with double branched cover the lens space $L(p,q)$. The links $L_{p/q}$ and $L_{p/(p-q)}$ are mirror images, and hence up to taking mirror images we may assume that precisely one of $p$ or $q$ is even. Then, following for example \cite{Ka-Survey}, one can write $p/q$ as a (positive) even continued fraction expansion
$$\frac{p}{q} = [2a_1, 2a_2, \ldots, 2a_n]^{+} := 2a_1 + \cfrac{1}{2a_2 + \cfrac{1}{\ddots \,\,  + \cfrac{1}{2a_n}}},$$

where $a_1,\ldots,a_n$ are integers. Then $L_{p/q}$ is the link shown in Figure \ref{twobridgestd}. If $n$ is even, $L_{p/q}$ is a knot, otherwise it is a two component link.

\begin{figure}[h!]
\begin{tikzpicture}
	\node[anchor=south west,inner sep=0] (image) at (0,0) {\includegraphics[width=0.7\textwidth]{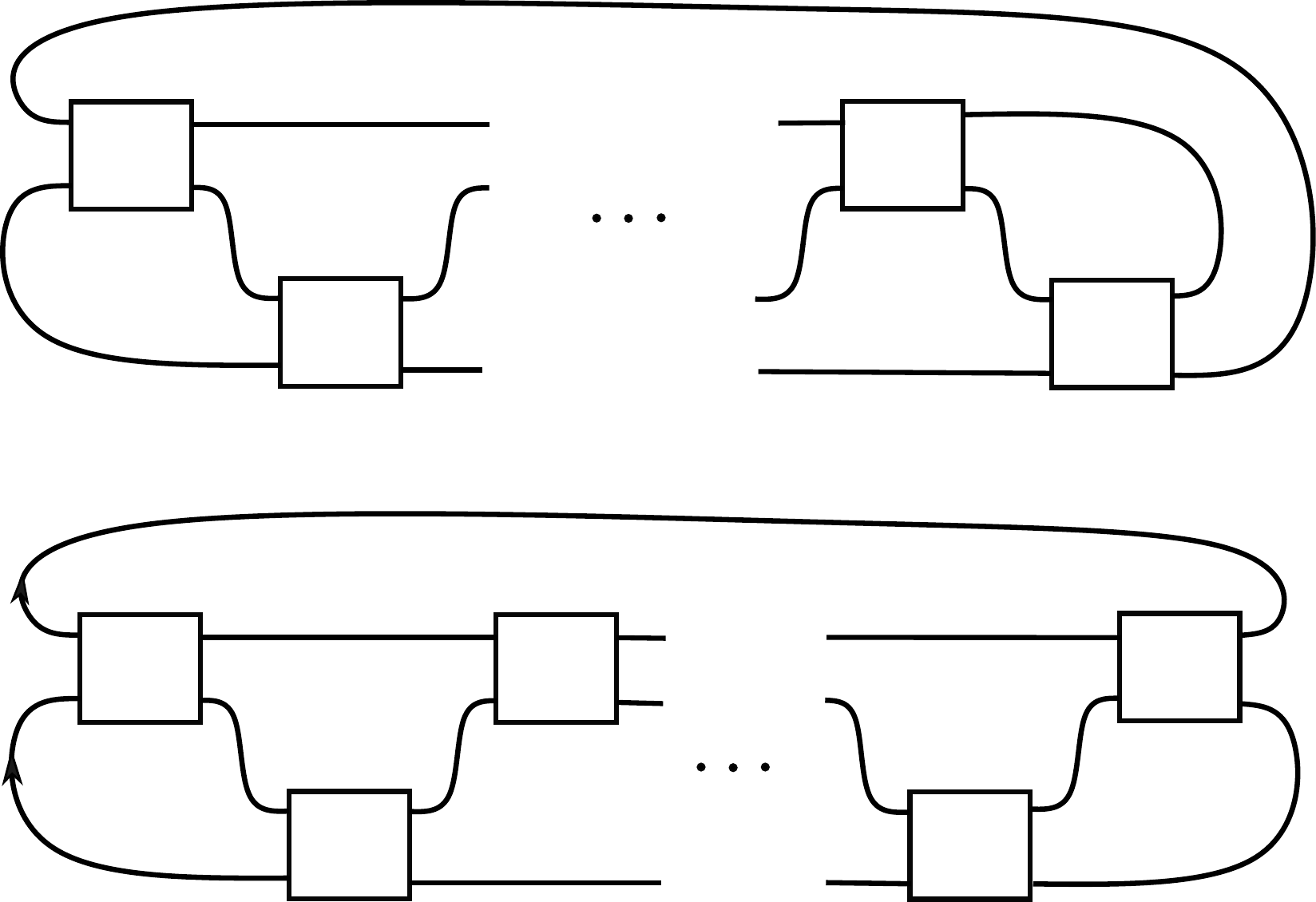}};
	\begin{scope}[x={(image.south east)},y={(image.north west)}]
        \node at (0.1,0.826) { \tiny $2a_1$};
        \node at (0.26, 0.63){\tiny $-2a_2$};
        \node at (0.685, 0.826){\tiny $2a_{n-1}$};
        \node at (0.84, 0.63){\tiny  $-2a_{n}$};
        
   \node at (0.11,0.256) { \tiny $2a_1$};
   \node at (0.425, 0.256){\tiny $2a_3$};
        \node at (0.265, 0.065){\tiny $-2a_2$};
        \node at (0.736, 0.065){\tiny $-2a_{n\!-\!1}$};
        \node at (0.9, 0.256){\tiny  $2a_{n}$};
    \end{scope}
\end{tikzpicture}
\caption{The top picture depicts a diagram for the two-bridge knot $K$ with fraction $[2a_1,\ldots, 2a_n]^{+}$ where $n$ is even. On bottom we have a diagram for the link $L$ with fraction $[2a_1,\ldots, 2a_n]^{+}$ where $n$ is odd, with a preferred orientation. In both diagrams the boxes should be replaced with the corresponding number of signed half-twists.}
\label{twobridgestd}
\end{figure}
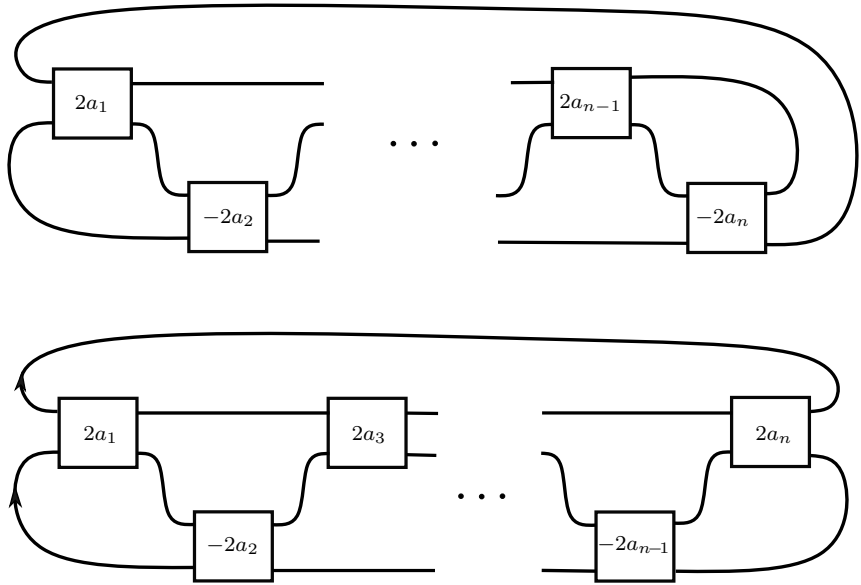

We are now ready to state a theorem which classifies all examples known to the authors of two-bridge links, all of whose branched covers are L-spaces. The statement restricted to two bridge knots with fraction $[2a_1,2a_2]^{+}$, and links with fraction $[2a_1]^{+}$ is due to Peters \cite{Pe09}. The case for two-bridge knots with fraction of the form $[2a_1,2a_2,\ldots, 2a_n]^{+}$ was shown by Teragaito \cite{Te14}. Again, two bridge knots are known to be $2$-periodic with quotient the unknot; this implies that their $n$-fold cyclic cover is the $2$-fold cover of a link $L_n$. This line of reasoning is known to experts; see e.g. \cite{MV-Cyclic2-bridge}. We observe that for a certain family of 2-bridge links, $L_n$ is non-split and alternating and conclude that $\Sigma_2(L_n)$ is an L-space.

\begingroup
\def\thetheorem{\ref{twobridgebranched}}
\begin{theorem}
Let $L$ be the two bridge link with fraction $[2a_1,2a_2,\ldots, 2a_n]^{+}$ where $a_i>0$ for all $i=1,\ldots, n$. In the case that $L$ is a two-component link we consider the link oriented as in Figure \ref{twobridgestd}. Then $\Sigma_n(L)$ is an L-space for all $n\geq 2$.
\end{theorem}
\addtocounter{theorem}{-1}
\endgroup

\begin{proof}
Figure \ref{twobridgesymmetric} shows a diagram for the link $L$, where $n$ is odd, which is rotationally symmetric about the origin (thinking of the diagram on the plane). An analogous process yields a symmetric diagram for the two bridge knot with fraction $[2a_1,\ldots, 2a_n]^+$ where $n$ is even.

\begin{figure}[h!]
\begin{tikzpicture}
	\node[anchor=south west,inner sep=0] (image) at (0,0) {\includegraphics[width=0.7\textwidth]{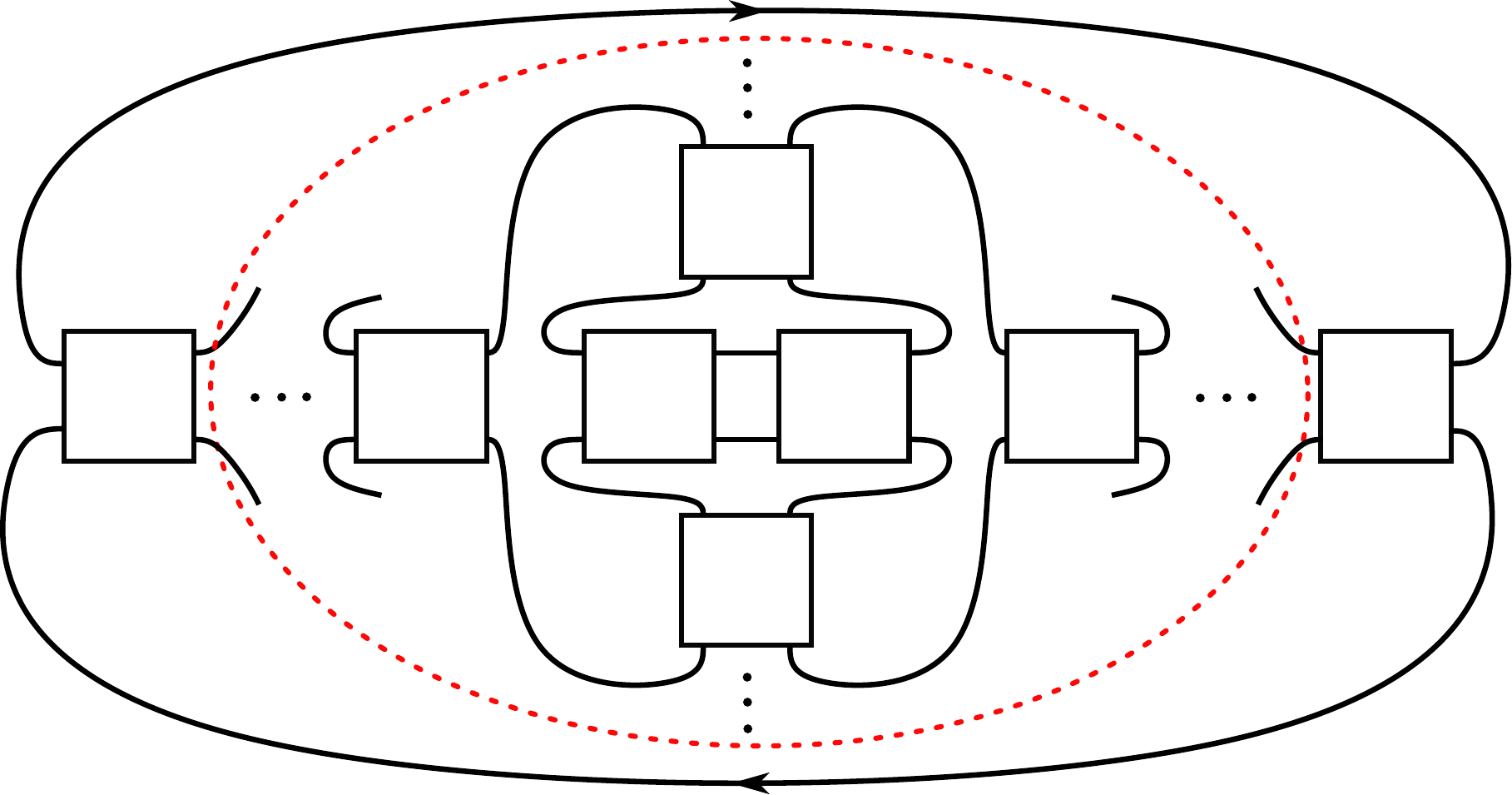}};
	\begin{scope}[x={(image.south east)},y={(image.north west)}]
        \node at (0.43,0.5) {  $a_1$};
        \node at (0.56,0.5) {  $a_1$};
        \node at (0.495,0.74) {  $-a_2$};
        \node at (0.495,0.26) {  $-a_2$};
        \node at (0.71, 0.5){ $a_3$};
        \node at (0.28, 0.5){ $a_3$};
         \node at (0.085, 0.5){ $a_n$};
        \node at (0.915, 0.5){ $a_n$};
        
    \end{scope}
\end{tikzpicture}
\caption{A symmetric diagram for the two bridge links with fraction $[2a_1,\ldots, 2a_n]^{+}$ where $n$ is odd. Imagining the diagram is placed on the plane, along the $x$-axis the twist boxes with coefficient $a_j$ for $j$ odd are placed at $(\pm j,0)$. Along the $y$-axis the twist boxes with coefficient $-a_j$ for $j$ even are placed at $(0,\pm j)$. The boxes are connected so that for each box indexed with $a_j$ is connected to both boxes indexed with $a_j+1$ in the manner depicted. The dotted circle indicates a Conway sphere that, when rotated through angle $\pi$, gives an isotopic link but transfers a half-twist from one box labelled $a_n$ to the other.}
\label{twobridgesymmetric}
\end{figure}
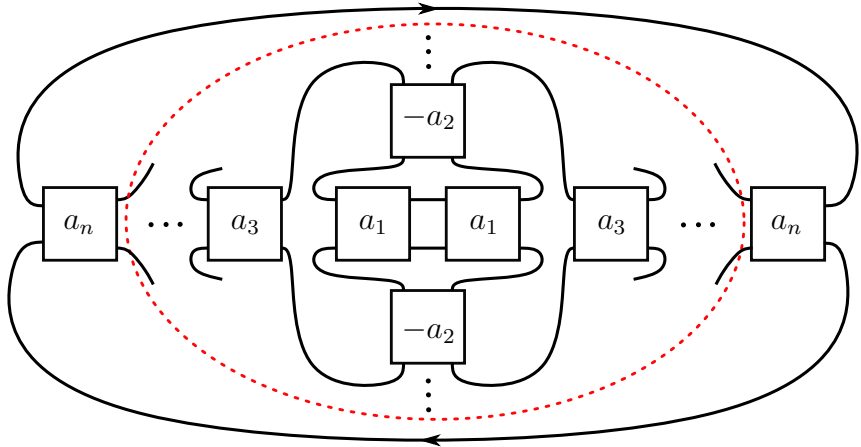

We sketch a proof that the diagrams in Figure \ref{twobridgestd} and Figure \ref{twobridgesymmetric} are isotopic. Using the Conway sphere indicated in Figure \ref{twobridgesymmetric} all of the half-twists of the left-hand twist box labelled $a_n$ can be untwisted at the expense of adding $a_n$ half twists to the box labelled $a_n$ on the right. 

There is a corresponding Conway sphere $S_i$ for each $2\leq i\leq n$ which contains all twist boxes labelled $a_j$ for $j<i$. Twisting along each $S_i$, either about its horizontal or vertical axis depending on the parity of $i$, in decreasing order from $S_n$ to $S_2$ yields the corresponding link in Figure \ref{twobridgestd}.

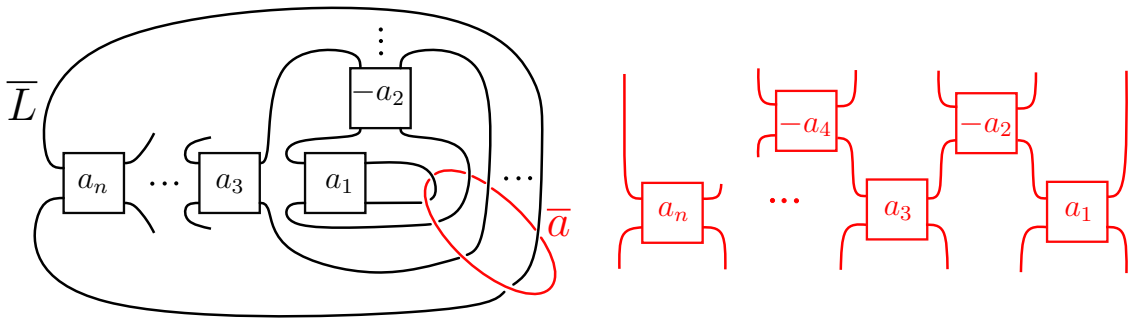
\begin{figure}[h!]
\begin{tikzpicture}
	\node[anchor=south west,inner sep=0] (image) at (0,0) {\includegraphics[width=0.9\textwidth]{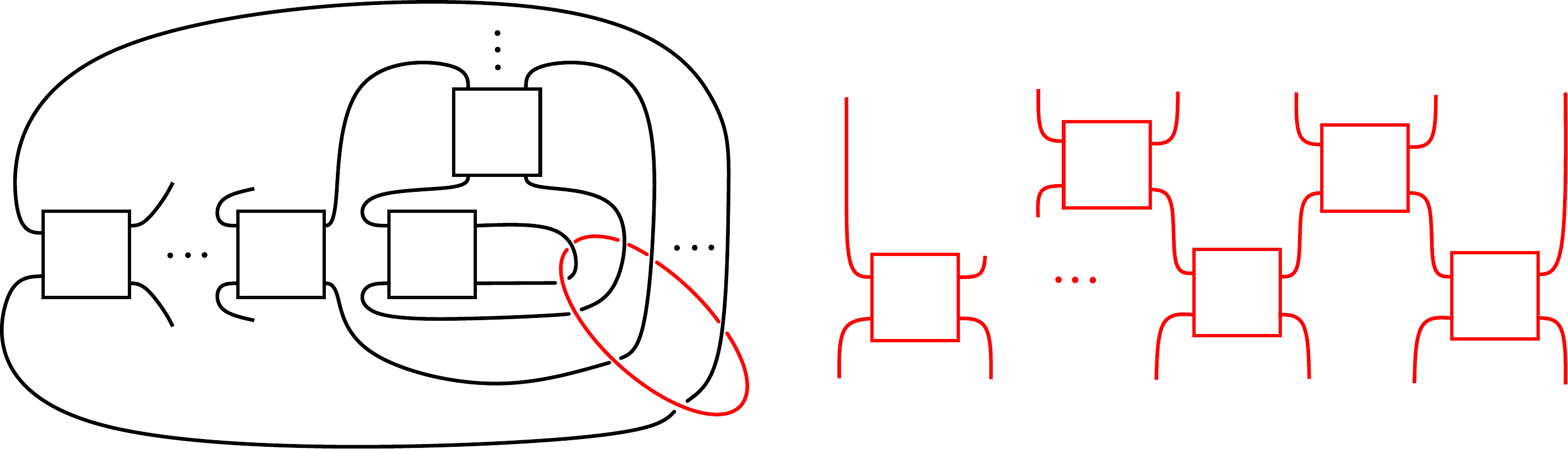}};
	\begin{scope}[x={(image.south east)},y={(image.north west)}]
        \node at (0.28, 0.43) {  $a_1$};
        \node at (0.585,0.342) {  $\textcolor{red}{a_n}$};
        \node at (0.315,0.72) {  $-a_2$};
        \node at (0.87,0.62) {  $\textcolor{red}{-a_2}$};
        \node at (0.79,0.345) {  $\textcolor{red}{a_3}$};
        \node at (0.705, 0.62){ $\textcolor{red}{-a_4}$};
        \node at (0.18, 0.43){ $a_3$};
         \node at (0.055, 0.43){ $a_n$};
        \node at (0.955, 0.342){ $\textcolor{red}{a_1}$};
         \node at (0.48,0.31) { \Large $\textcolor{red}{\overline{a}}$};
          \node at (-.01,0.7) { \Large ${\overline{L}}$};
    \end{scope}
\end{tikzpicture}
\caption{The link $\overline{L}\cup \overline{a}$ on the left. This link is symmetric; exchanging the roles of  $\overline{L}$ and $\overline{a}$, and then cutting $\overline{a}$ along the obvious disk bounded by $\overline{L}$ yields the picture on the right.}
\label{twobridgequotient}
\end{figure}

Since this diagram is symmetric about the origin, we can take the quotient of $(S^3,L\cup a)$ under the rotation about $a$ through angle $\pi$ where  $a$ is the axis in $S^3$ perpendicular to the origin of the diagram, yielding $(S^3, \overline{L}\cup\overline{a})$; see Figure \ref{twobridgequotient}. It is not hard to see that $\overline{L}$ is unknotted. If we now assume that each $a_i>0$ we also see that the quotient diagram of $\overline{L}$ is alternating and in fact the diagram can be modified by only Reidemeister 1 moves to obtain a diagram for $\overline{L}$ as the standard unknot. It follows that the link $\overline{L}\cup\overline{a}$ is symmetric, meaning that for any diagram there is an isotopy which exchanges the roles of $\overline{L}$ and $\overline{a}$, but otherwise leaves the picture unchanged.

As in the proof of Proposition \ref{proplinkL_n}, we see that we can now express $\Sigma_n(L)$ as $\Sigma_2(L_n)$ where $L_n$ is the lift of $\overline{a}$ under the $n$-fold branched cover of the unknotted $\overline{L}$. Since $\overline{a}$ has an alternating diagram as a pattern in the complement of $L$, its lift will also be alternating. Thus each $\Sigma_n(L)$ is homeomorphic to the double-branched cover of a non-split alternating link, and hence is an L-space.
\end{proof}

\bibliography{Pretzelbibliography}

\end{document}